\documentclass[a4paper,10pt]{amsproc}
\title{Algebraic properties of the ring $C(X)_\mathcal{P}$}
\usepackage{amsfonts, amsmath, amssymb, graphicx, mathrsfs}
\usepackage[all]{xy}

\newdir{ >}{!/6pt/@{ }*:(1,0.2)@_{>}*:(1,-0.2)@^{>}}

\theoremstyle{plain}

\newtheorem{theorem}{Theorem}[section]
\newtheorem{lemma}[theorem]{Lemma}

\newtheorem{proposition}[theorem]{Proposition}
\theoremstyle{definition}
\newtheorem{definition}[theorem]{Definition}
\theoremstyle{definitions}
\newtheorem{definitions}[theorem]{Definitions}

\newtheorem{remark}[theorem]{Remark}
\newtheorem{counter example}[theorem]{Counter Example}
\newtheorem{notation}[theorem]{Notation}

\newtheorem{corollary}[theorem]{Corollary}
\newtheorem{observation}[theorem]{observation}

\newtheorem{example}[theorem]{Example}

\numberwithin{equation}{section}

\author[A. Dey]{Amrita Dey}	\address{Department of Pure Mathematics, University of Calcutta, 35, Ballygunge Circular Road, Kolkata 700019, West Bengal, India}	\email{deyamrita0123@gmail.com}

\author[S. Bag]{Sagarmoy Bag}	\address{Department of Mathematics, Bankim Sardar College, Tangrakhali, West Bengal 743329, India}	\email{sagarmoy.bag01@gmail.com}

\author[D. Mandal]{Dhananjoy Mandal} \address{Department of Pure Mathematics, University of Calcutta, 35, Ballygunge Circular Road, Kolkata 700019, West Bengal, India}  \email{dmandal.cu@gmail.com}


\begin{document}
	
	\title{Algebraic properties of the ring $C(X)_\mathcal{P}$.}

	%

	\thanks {}
	
	\maketitle
	
	\large
	\begin{abstract}
		Our aim is to study certain algebraic properties of the ring $C(X)_\mathcal{P}$ of real-valued functions on $X$ whose closure of discontinuity set is in an ideal of closed sets. We characterize $\mathcal{P}P$-spaces using $z$-ideals and essential ideals of $C(X)_\mathcal{P}$ and also almost $\mathcal{P}P$-spaces using $z^0$-ideals of $C(X)_\mathcal{P}$ and a topology finer than the original topology on $X$. We deduce that each maximal ideal of $C(X)_F$ \cite{GGT2018} (resp. $T'(X)$  \cite{A2010}) is a $z^0$-ideal. We establish that the notions of clean ring, weakly clean ring, semiclean ring, almost clean ring and exchange ring coincide in the ring $C(X)_\mathcal{P}$.
		End of this paper, we also characterize $\mathcal{P}P$-spaces and almost $\mathcal{P}P$-spaces using certain ideals having depth zero. We exhibit a condition on $\mathcal{P}$ under which prime and essential ideals of $C(X)_\mathcal{P}$ have depth zero.
	\end{abstract}

	\section{Introduction}
	
	Let $\mathcal{P}$ be an ideal of closed subsets of a $T_1$ topological space $(X,\tau)$ and we call the triplet $(X,\tau,\mathcal{P})$ a $\tau \mathcal{P}$-space. The subring $C(X)_\mathcal{P}$ was introduced in \cite{DABM} as the subring of $\mathbb{R}^X$ consisting of those functions $f\colon X\longrightarrow \mathbb{R}$ which satisfy the condition $\overline{D_f}\in \mathcal{P}$, where $D_f=\{x\in X\colon f\text{ is discontinuous at }x \}$. 
	For a $\tau \mathcal{P}$-space $(X,\tau,\mathcal{P})$, it is easy to check that the collection $coz_\mathcal{P}[X]= \{coz(f)\colon f\in C(X)_\mathcal{P} \}$  of all cozero sets of functions in $C(X)_\mathcal{P}$, forms a base for open sets for a topology $\tau_\mathcal{P}$ on the set $X$, which is in general finer than $\tau$. However, the two topologies, viz, $\tau_\mathcal{P}$ and $\tau$ coincide if and only if $C(X)=C(X)_\mathcal{P}$. We use the following lemma to achieve this equivalence.
\begin{lemma} \cite{DABM} \label{l00}
	Let $\mathcal{P}$ be an ideal of all closed subsets of a $T_1$-space $X$ and $\mathcal{P}'$ be the ideal of all closed subsets of the set of isolated points of $X$. Then $C(X)_\mathcal{P}=C(X)$ if and only if $\mathcal{P}\subseteq \mathcal{P'}$.
\end{lemma}
The following theorem is now obvious.
\begin{theorem} \label{t0}	
	For a $\tau \mathcal{P}$-space $(X,\tau, \mathcal{P})$, $\tau =\tau_\mathcal{P}$ if and only if $C(X)=C(X)_\mathcal{P}$.
\end{theorem} 
\begin{notation}
	We simply write $X_\mathcal{P}$ for the topological space $(X,\tau_\mathcal{P})$.
\end{notation}

One may note that if $\chi_{\{{a}\}}\in C(X)_\mathcal{P}$ for all $a\in X$, $\tau_\mathcal{P}$ is nothing but the discrete topology on
 $X$. So in this case $C(X_\mathcal{P})=\mathbb{R}^X$ which is often not equal to $C(X)_\mathcal{P}$ (for example, take $X=\mathbb{R}$, the real line and $\mathcal{P}=\mathcal{P}_f$ \cite{DABM}). Moreover, if $X=\mathbb{R}$ endowed with cofinite topology and $\mathcal{P}=\{\emptyset,\{1\} \}$, then $(X,\tau)$ is a connected topological space but $(X,\tau_\mathcal{P})$ contains a clopen set $\{1\}$ and is thus disconnected.

Throughout the article, for a subset $A\subseteq X$, we shall denote $int_{X_\mathcal{P}}A$ to be the interior of $A$ in $X_\mathcal{P}$ and $cl_{X_\mathcal{P}}A$ to be the closure of $A$ in $X_\mathcal{P}$.

The aim of  this article is to study certain algebraic properties of $C(X)_\mathcal{P}$ starting with the study of various types of ideals of the ring.

 In Section 2, we study $z$-ideal, $z^0$-ideal and essential ideal of $C(X)_\mathcal{P}$. We have introduced the concept of  $z_\mathcal{P}$-ideals of $C(X)_\mathcal{P}$ in \cite{DABM} and in this section we establish  that $z_\mathcal{P}$-ideals of $C(X)_\mathcal{P}$ and  $z$-ideals of $C(X)_\mathcal{P}$ are equivalent (Theorem \ref{t1.3}). We characterize $\mathcal{P}P$-spaces using $z$-ideals of $C(X)_\mathcal{P}$ (Theorem \ref{l1.0}). We obtain that the nature of $z$-ideals of $C(X)_\mathcal{P}$ depends heavily on the choice of $\mathcal{P}$ (See Example \ref{eg}). Furthermore, we deduce that each ideal of $T'(X)$ \cite{A2010} is a $z$-ideal. In this section, we also establish a topological characterization of $z^0$-ideals (\cite{AKA1999}) of $C(X)_\mathcal{P}$ using $X_\mathcal{P}$. Using this characterization, we show that all maximal ideals of the rings $C(X)_F$ \cite{GGT2018}, $T'(X)$ \cite{A2010} and $C(X)_K$ \cite{DABM} are $z^0$-ideals. We  introduce the notion of almost $\mathcal{P}P$-spaces and characterize these spaces algebraically using $z$-ideals and $z^0$-ideals of $C(X)_\mathcal{P}$ (Theorem \ref{t1.9}). We end this section by discussing essential ideals of $C(X)_\mathcal{P}$. We characterize these ideals using the topology $X_\mathcal{P}$ and $\mathcal{P}P$-spaces using essential ideals of $C(X)_\mathcal{P}$.



	In Section 3, we discuss conditions under which the ring  $C(X)_\mathcal{P}$ becomes a clean ring, semiclean ring, almost clean ring, weakly clean ring and even exchange ring. The first interesting observation we make is that the rings $C(X)_F$, $T'(X)$ and $C(X)_K$ are clean rings provided $X$ contains at least three elements. We characterize clean elements and almost clean elements of $C(X)_\mathcal{P}$ and even clean elements of $C^*(X)_\mathcal{P}$, using the topology $X_\mathcal{P}$. We introduce the notion of a $\tau_\mathcal{P}$-zero-dimensional space and use it to characterize $C(X)_\mathcal{P}$ as a clean ring. We obtain that a Tychonoff, extremally disconnected space is $\tau \mathcal{P}$-zero-dimensional, for any choice of $\mathcal{P}$. We infer that $\beta \mathbb{N}$ is $\tau \mathcal{P}$-zero-dimensional and $C(\beta \mathbb{N})_\mathcal{P}$ is a clean ring, for any choice of $\mathcal{P}$.  We end this section by establishing that the notion of clean ring, weakly clean ring, semiclean ring, almost clean ring and exchange ring, all coincide in case of the ring $C(X)_\mathcal{P}$.
	
	In the final section of this article, we study the depths of certain ideals of $C(X)_\mathcal{P}$. We characterize $\mathcal{P}P$-spaces and almost $\mathcal{P}P$-spaces using depth of ideals and essential ideals of  $C(X)_\mathcal{P}$ respectively (Theorem \ref{t5.1}). We obtain  that the depth of each free ideal (and hence essential ideal), each prime ideal and each maximal ideal of $C(X)_\mathcal{P}$ is zero whenever $C(X)_\mathcal{P}\supseteq \{\chi_{\{{x}\}}\colon x\in X \}$ and as a consequences  we observe  the same is true for the rings  $C(X)_F$, $C(X)_K$ and $T'(X)$.  We can even say that the depth of each prime ideal of $C(X)_F$ (also, $C(X)_K$ and $T'(X)$) is zero. We note that $C(X)$ may have free ideals or maximal ideals having non-zero depths. At the end of this section we point out that depth of all ideals of $C(X)_\mathcal{P}$ may not be zero even if $C(X)_\mathcal{P}\supseteq \{\chi_{\{{x}\}}\colon x\in X \}$.


	\section{Ideals of $C(X)_\mathcal{P}$.}
	
	In this section, we deal with various types of ideals of $C(X)_\mathcal{P}$ which include $z$-ideal, $z^0$-ideal and essential ideal of $C(X)_\mathcal{P}$.
	
	Let $R$ be a commutative ring with unity. We denote $Max(R)$ and $P(R)$ to be the set of all maximal ideals of $R$ and the set of all minimal prime ideals of $R$ respectively. For each $a\in R$, we define $M(a)=\bigcap \{M\in Max(R)\colon a\in M \}$ and $P(a)=\bigcap \{I\in P(R)\colon a\in I \}$. An ideal $I$ of $R$ is said to be $z$-ideal if it satisfies the condition: $a\in I\implies M(a)\subseteq I$ \cite{Mason} and it is said to be a $z^0$-ideal if it satisfies the condition: $a\in I\implies P(a)\subseteq I$ \cite{AKA1999}. An ideal $E$ of $R$ is said to be an essential ideal if it intersects all non zero ideals of $R$ non-trivially. Intersection of all essential ideals of $R$ is said to be the $Socle(R)$.  
	
	We first establish a characterization of essential ideals of an arbitrary reduced ring $R$. (We note that $C(X)_\mathcal{P}$ is a reduced ring for any $\tau \mathcal{P}$-space $(X,\tau, \mathcal{P})$.)
	\begin{lemma} \label{l3.0}
		Let $E$ be a non-zero ideal of a reduced ring $R$. Then $E$ is an essential ideal of $R$ if and only if $Ann(E)=\{0\}$, where $Ann(E)=\{r\in R\colon rx=0 \text{ for all }x\in E \}$.
	\end{lemma}
	 \begin{proof}
	 	Let us assume $E$ to be an essential ideal of $R$ and let $a\in Ann(E)\cap E$. Then $a^2=0\implies a=0$, as $R$ is a reduced ring. So $Ann(E)\cap E=\{0\}$. Thus $Ann(E)=\{0\}$. Conversely, let $Ann(E)=\{0\}$ and $I$ be a non-zero ideal of $R$. Also let $x\in I\setminus E$. Then there exists $r\in E$ such that $rx\neq 0$. So $0\neq rx\in I\cap E$. Thus $E$ is an essential ideal of $R$.
	 \end{proof}

	\subsection{$z$-ideals of $C(X)_\mathcal{P}$.}

	It is known that for a Tychonoff space $X$, an ideal $I$ of $C(X)$ is a $z$-ideal if and only if for $f,g\in C(X)$, $Z(f)\subseteq Z(g)$ with $f\in I$ implies that $g\in I$ (See 4A(5) \cite{GJ1976}). Also it is proved in \cite{AKA1999} that an ideal $I$ of $C(X)$ is a $z^0$-ideal if and only if for $f,g\in C(X)$, $int$ $Z(f) = int$ $Z(g)$ and $f\in I$ implies that $g\in I$.
	
	\begin{definition} \cite{DABM} \label{d1.1}
		An ideal $I$ of $C(X)_\mathcal{P}$ is called a $z_\mathcal{P}$-ideal if $Z_\mathcal{P}^{-1}Z_\mathcal{P}[I]=I$. Equivalently, an ideal $I$ of $C(X)_\mathcal{P}$ is a $z_\mathcal{P}$-ideal if and only if for $f,g\in C(X)_\mathcal{P}$, $Z_\mathcal{P}(f)\subseteq Z_\mathcal{P}(g)$ with $f\in I$ implies that $g\in I$.
	\end{definition}
	
	We attempt to show that an ideal $I$ of $C(X)_\mathcal{P}$ is a $z_\mathcal{P}$-ideal if and only if $I$ is a $z$-ideal of $C(X)_\mathcal{P}$. For this, we need the following lemma.
	
	\begin{lemma} \label{l1.2}
		If $f\in M(g)$, then $Z_\mathcal{P}(g)\subseteq Z_\mathcal{P}(f)$.
	\end{lemma}
	\begin{proof}
		If possible, let there exists $f\in M(g)$ such that $Z_\mathcal{P}(g)\nsubseteq Z_\mathcal{P}(f)$. Then there exists $x\in Z_\mathcal{P}(g)\setminus Z_\mathcal{P}(f)$. So $g\in M_x=\{f\in C(X)_\mathcal{P}\colon f(x)=0 \}$ (\cite{DABM}) but $f\notin M_x$ which contradicts that $f\in M(g)$.
	\end{proof}
	We now state and prove an equivalence between $z$-ideals and $z_\mathcal{P}$-ideals of $C(X)_\mathcal{P}$.
	\begin{theorem} \label{t1.3}
		An ideal $I$ of $C(X)_\mathcal{P}$ is a $z_\mathcal{P}$-ideal if and only if $I$ is a $z$-ideal of $C(X)_\mathcal{P}$.
	\end{theorem}
	\begin{proof}
		Let $I$ be a $z$-ideal of $C(X)_\mathcal{P}$ and $Z_\mathcal{P}(f)\in Z_\mathcal{P}[I]$ for some $f\in C(X)_\mathcal{P}$. Then there exists $g\in I$ such that $Z_\mathcal{P}(f)=Z_\mathcal{P}(g)$. Let $M$ be an arbitrary maximal ideal of $C(X)_\mathcal{P}$ containing $g$. Then $f\in Z_\mathcal{P}^{-1}Z_\mathcal{P}[M]=M$. Since $M$ is an arbitrary maximal ideal of $C(X)_\mathcal{P}$ containing $g$, $f\in M(g)$ and $M(g)\subseteq I$ follows from our assumption. 
		
		Conversely, let $I$ be a $z_\mathcal{P}$-ideal of $C(X)_\mathcal{P}$ and let $f\in I$. We need to show that $M(f)\subseteq I$. Let $g\in M(f)$. Then by Lemma \ref{l1.2}, $Z_\mathcal{P}(f)\subseteq Z_\mathcal{P}(g)$ and since $f\in I$, we have $g\in I$. This completes the proof.
	\end{proof}
	
	Thus, from now onwards, we use the term $z$-ideal to describe a $z_\mathcal{P}$-ideal of $C(X)_\mathcal{P}$. 
	
	We note that the nature of $z$-ideals of $C(X)_\mathcal{P}$ depends heavily on the choice of $\mathcal{P}$ as can be seen in the following example.
	\begin{example} \label{eg}
		Consider $X=\mathbb{R}$.
		\begin{enumerate}
			\item For any ideal of closed sets $\mathcal{P}$ containing $\{0\}$, the ideal generated by the identity function is a $z$-ideal of $C(X)_\mathcal{P}$ but not of $C(X)$.
			\item Consider the function $f(x)=\sin \pi x $ for all $x\in X$ and the ideal $I=<f^2>$. Then $I$ is not a $z$-ideal of $C(X)_\mathcal{P}$ for $\mathcal{P}=\{\emptyset \}$, $\mathcal{P}=\mathcal{P}_f$ or even $\mathcal{P}=\mathcal{K}$. But $I$ is a $z$-ideal of $C(X)_\mathcal{P}$ if we take $\mathcal{P}=\mathcal{P}_{nd}$. Indeed if $I$ is a $z$-ideal of $C(X)_\mathcal{P}$, then there exists $g\in C(X)_\mathcal{P}$ such that $f=f^2g$ and so $D_g=\mathbb{Z}$. Thus $\mathbb{Z}\in \mathcal{P}$.
		\end{enumerate}
		
	\end{example}

	We need the following lemma to establish a characterization for $\mathcal{P}P$-spaces. We recall that a $\tau \mathcal{P}$-space $(X,\tau,\mathcal{P})$ is said to be a $\mathcal{P}P$ space \cite{DABM} if the ring $C(X)_\mathcal{P}$ is a Von-Neumann regular ring. 
	\begin{lemma} \label{l01}
		A function $f\in C(X)_\mathcal{P}$ is a Von-Neumann regular element (for a commutative ring $R$ with unity, an element $a\in R$ is said to be a Von-Neumann regular element if there exists $b\in R$ such that $a=a^2b$) if and only if $Z_\mathcal{P}(f)=coz(g) $ for some $g\in C(X)_\mathcal{P}$.
	\end{lemma}
	
	\begin{proof}
		Let $f\in C(X)_\mathcal{P}$ be a Von-Neumann regular element. Then there exists $g\in C(X)_\mathcal{P}$ such that $f=f^2g$ which implies that $Z_\mathcal{P}(f)=coz(\boldsymbol{1}-fg)$. Conversely, let $Z_\mathcal{P}(f)=coz(g) $ for some $g\in C(X)_\mathcal{P}$. Define $h\colon X\longrightarrow \mathbb{R}$ by $h(x)=\begin{cases}
			\dfrac{1}{f(x)} \text{ if }x\in coz(f)=Z_\mathcal{P}(g) \\ 0 \text{ if }x\in Z_\mathcal{P}(f)
		\end{cases}.$ Then $D_h\subseteq \overline{D_f}\cup \overline{D_g}\in \mathcal{P}$ which ensures that $h\in C(X)_\mathcal{P}$ and $f=f^2h$. Thus $f$ is a Von-Neumann regular element.
	\end{proof} The following corollary is an immediate consequence of above result. 
	\begin{corollary} \label{c2}
		A $\tau \mathcal{P}$-space $(X,\tau,\mathcal{P})$ is a $\mathcal{P}P$-space if and only if $Z_\mathcal{P}[X]=coz_\mathcal{P}[X]$.
	\end{corollary}
	We are now able to characterize $\mathcal{P}P$-spaces as follows.
	\begin{theorem} \label{l1.0}
		A $\tau \mathcal{P}$-space $(X,\tau,\mathcal{P})$ is a $\mathcal{P}P$-space if and only if every ideal of $C(X)_\mathcal{P}$ is a $z$-ideal.
	\end{theorem}
	\begin{proof}
		Let $(X,\tau,\mathcal{P})$ be a $\mathcal{P}P$-space and $I$ an ideal of $C(X)_\mathcal{P}$. Let $f\in C(X)_\mathcal{P}$ with $Z_\mathcal{P}(f)\in Z_\mathcal{P}[I]$. Then there exists $g\in I$ such that $Z_\mathcal{P}(f)=Z_\mathcal{P}(g)$. Since $f$ is Von-Neumann regular, there exists $h\in C(X)_\mathcal{P}$ such that $Z_\mathcal{P}(f)=coz(h)$ (by Lemma \ref{l01}). Define $\displaystyle{t(x)=\begin{cases}
				\dfrac{f(x)}{g(x)} \text{ if }x\notin Z_\mathcal{P}(g) \\
				0 \text{ otherwise }
		\end{cases}}.$ Then $t$ is continuous on $X\setminus (\overline{D_f}\cup \overline{D_g}\cup \overline{D_h})$ and so $t\in C(X)_\mathcal{P}$ with $f=gt\in I$.\\ Conversely, let us assume that each ideal of $C(X)_\mathcal{P}$ is a $z$-ideal and let $f\in C(X)_\mathcal{P}$ be a non-unit element. Then $<f^2>$ is an ideal of $C(X)_\mathcal{P}$ and hence a $z$-ideal. Since $Z_\mathcal{P}(f)=Z_\mathcal{P}(f^2)$, it follows that $f\in <f^2>$. Thus, $f=f^2g$ for some $g\in C(X)_\mathcal{P}$. This shows that $(X,\tau,\mathcal{P})$ is a $\mathcal{P}P$-space. 
	\end{proof}
	
	\begin{remark}\label{rem1}
		The ring $T'(X)$ is a Von-Neumann Regular ring (See \cite{A2010}) and $T'(X)=C(X)_{\mathcal{P}_{nd}}$ (See \cite{DABM}). This ensures that each ideal of $T'(X)$ is a $z$-ideal.
		
	\end{remark}

	\subsection{$z^0$-ideals of $C(X)_\mathcal{P}$.
	}
	In what follows, we establish a result analogous to Proposition 2.2 of \cite{AKA1999}. We need the following results to achieve this.

	\begin{lemma} \label{l1.5}
		Let $f,g\in C(X)_\mathcal{P}$. Then $int_{X_\mathcal{P}}(Z_\mathcal{P}(f))\subseteq int_{X_\mathcal{P}}(Z_\mathcal{P}(g))$ if and only if $Ann(f)\subseteq Ann(g)$.
	\end{lemma}

	The following results are established in \cite{AKA2000} for an arbitrary reduced ring $R$.
	
	\begin{proposition} \label{p1.6}
		Let $I$ be a proper ideal in a reduced ring $R$. Then the following statements are equivalent.
		\begin{enumerate}
			\item $I$ is a $z^0$-ideal.
			\item $P(a)=P(b)$ and $b\in I$, imply that $a\in I$.
			\item $Ann(a)=Ann(b)$ and $b\in I$, imply that $a\in I$.
			\item $a\in I$ implies that $Ann(Ann(a))\subseteq I$.
			
		\end{enumerate}
	\end{proposition}

	\begin{lemma} \cite{AKA2000} \label{l1.4}
		Let $R$ be a reduced ring and $a\in R$. Then we have \[P(a)=\{b\in R\colon Ann(a)\subseteq Ann(b) \}. \]
	\end{lemma}

	Using above results, we characterize $z^0$-ideals of the reduced ring $C(X)_\mathcal{P}$.
	
	\begin{theorem} \label{t1.6}
		If $I$ is a proper ideal of $C(X)_\mathcal{P}$, then the following statements are equivalent.
		\begin{enumerate}
			\item $I$ is a $z^0$-ideal.
			\item $P(f)=P(g)$ and $g\in I$, imply that $f\in I$.
			\item $Ann(f)=Ann(g)$ and $g\in I$, imply that $f\in I$.
			\item $f\in I$ implies that $Ann(Ann(f))\subseteq I$.
			\item $int_{X_\mathcal{P}}(Z_\mathcal{P}(f))= int_{X_\mathcal{P}}(Z_\mathcal{P}(g))$ and $f\in I$, imply that $g\in I$.
			\item If $S\subseteq I$ is a finite set and $Ann(S)\subseteq Ann(g)$, then $g\in I$.
		
		\end{enumerate}
	\end{theorem}
	
	\begin{proof}
			It follows from Proposition \ref{p1.6} that statements $(1)$ to $(4)$ are equivalent. Also by Lemma \ref{l1.5}, Statements $(3)$ and $(5)$ are equivalent. We next show that $(4)$ implies $(6)$. We assume $(4)$ is true and let $Ann(S)\subseteq Ann(g)$, where $S\subseteq I$ is a finite set. Then $S=\{f_1,f_2,\cdots,f_n \}$, where $f_i\in I$ for all $i=1,2,\cdots,n$ for some $n\in \mathbb{N}$ which implies that $Ann(S)=Ann(k)$, where $k=f_1^2+f_2^2+\cdots +f_n^2\in I$. By $(4)$, $Ann(Ann(k))\subseteq I$. It is clear from $Ann(S)\subseteq Ann(g)$ that $g\in Ann(Ann(k))$. So $g\in I$. This proves $(6)$. Finally we assume that $(6)$ holds and let $f\in I$ and $g\in Ann(Ann(f))$. Then $Ann(f)\subseteq Ann(g)$. But $f\in I$ and hence by $(6)$, $g\in I$.
	\end{proof}
	
	The following definition can be found in \cite{H1988}.
	\begin{definition} \label{d1.7}
		A ring $R$ is said to satisfy property A if each finitely generated ideal of $R$ consisting of zero divisors has a non-zero annihilator.
	\end{definition}
	If $I=<f_1,f_2,\cdots, f_n>$ is a finitely generated ideal of $C(X)_\mathcal{P}$ consisting of zero divisors. Then it can be easily seen that $Ann(I)=Ann(f)$, where $f=f_1^2+f_2^2+\cdots f_n^2\in I$. This ensures that $C(X)_\mathcal{P}$ satisfies the property A. Thus any ideal $I$ of $C(X)_\mathcal{P}$ consisting of zero divisors is contained in a $z^0$-ideal (See Theorem 1.21 \cite{AKA2000}). It follows that a maximal ideal $M$ of $C(X)_\mathcal{P}$ consisting of zero divisors is a $z^0$-ideal. Also since arbitrary intersection of $z^0$-ideals of a ring is itself a $z^0$-ideal, if $I$ is an ideal of $C(X)_\mathcal{P}$ consisting of zero divisors, then there is a smallest $z^0$-ideal containing $I$.
	
	One must note that if $C(X)_\mathcal{P}$ contains all the characteristic functions $\chi_{\{{a}\}}$, $a\in X$ (when $\mathcal{P}$ contains all singleton subsets of $X$), each element of $C(X)_\mathcal{P}$ is either a unit or a zero divisor. Thus, in this situation, every maximal ideal of $C(X)_\mathcal{P}$ is a $z^0$-ideal of $C(X)_\mathcal{P}$. In particular every maximal ideal of the rings $C(X)_F$ (\cite{GGT2018}), $T'(X)$ (\cite{A2010} ) and $C(X)_K$ (\cite{DABM} ) are $z^0$-ideals of the respective rings.
	
	Let $I$ be a $z^0$-ideal and $P$ be a prime ideal of $C(X)_\mathcal{P}$ which is minimal over $I$. Then we can show that $P$ is a $z^0$-ideal of $C(X)_\mathcal{P}$. To see this, let $int_{X_\mathcal{P}}Z_\mathcal{P}(f)=int_{X_\mathcal{P}}Z_\mathcal{P}(g)$, where $f\in I$ and $g\in C(X)_\mathcal{P}$. If $f\in I$, then by Theorem \ref{t1.6} $g\in I$. If possible, let $f\in P\setminus I$ and $g\notin P$. We define $S=(C(X)_\mathcal{P})\cup \{hf^n\colon h\notin P, n\in \mathbb{N} \}$. Then $S$ is closed under multiplications and is disjoint from $I$, since if $hf^n\in I$ for some $h\notin P$ and $n\in \mathbb{N}$, then $hg\in I$ which is not possible. Therefore there exists a prime ideal $P_1$ of $C(X)_\mathcal{P}$ containing $I$ and disjoint from $S$. It follows that $P_1\subsetneq P$ which contradicts that $P$ is a prime ideal minimal over $I$. Hence a minimal prime ideal of $C(X)_\mathcal{P}$ is a $z^0$-ideal of $C(X)_\mathcal{P}$.
	
	Moreover, it can be easily seen that if $S$ is a subset of $C(X)_\mathcal{P}$ which is regular closed in $X_\mathcal{P}$, that is $cl_{X_\mathcal{P}}(int_{X_\mathcal{P}}(S))=S$, then $M_S=\{f\in C(X)_\mathcal{P}\colon S\subseteq Z_\mathcal{P}(f) \}$ is a $z^0$-ideal of $C(X)_\mathcal{P}$. To show this, it is sufficient to show that $S\subseteq Z_\mathcal{P}(g)$ when $int_{X_\mathcal{P}}(Z_\mathcal{P}(f))= int_{X_\mathcal{P}}(Z_\mathcal{P}(g))$ with $f\in I$. If not, then there exists $x\in S\cap coz_\mathcal{P}(g)$. Since $coz_\mathcal{P}(g)$ is a neighbourhood of $x$ in $X_\mathcal{P}$ and $S$ is regular closed in $X_\mathcal{P}$, there exists $t\in coz_\mathcal{P}(g)\cap int_\mathcal{P}S$ and $int_\mathcal{P} S\subseteq int_{X_\mathcal{P}}(Z_\mathcal{P}(f))= int_{X_\mathcal{P}}(Z_\mathcal{P}(g))\subseteq Z_\mathcal{P}(g)$. Thus $t\in Z_\mathcal{P}(g)$ which is a contradiction.
	
	Another easy yet interesting observation is that every $z^0$-ideal of $C(X)_\mathcal{P}$ is a $z$-ideal. However the converse may not be true. Consider $X=[0,1]\cup \{2\}$ endowed with the subspace topology of the real line and $\mathcal{P}=\{ \{2\},\emptyset \}$. Then $C(X)_\mathcal{P}=C(X)$ and $M_0=\{f\in C(X)_\mathcal{P}\colon f(0)=0 \}$ is a $z$-ideal of $C(X)_\mathcal{P}$ which is not a $z^0$-ideal. This makes us ask the following question: When is every $z$-ideal of $C(X)_\mathcal{P}$ a $z^0$-ideal? To answer this, we introduce the notion of almost $\mathcal{P}P$-spaces.
	
	\begin{definition} \label{d1.8}
		A $\tau \mathcal{P}$-space $(X,\tau,\mathcal{P})$ is said to be an almost $\mathcal{P}P$-space if the zero set of every non-unit member of $C(X)_\mathcal{P}$ has non-empty interior in $X_\mathcal{P}$. That is whenever $Z_\mathcal{P}(f)\neq \emptyset$ for some $f\in C(X)_\mathcal{P}$, $int_{X_\mathcal{P}}(Z_\mathcal{P}(f))\neq \emptyset$. 
	\end{definition}
	It can be easily seen from the following lemma that a $\tau \mathcal{P}$-space $(X,\tau,\mathcal{P})$ is an almost $\mathcal{P}P$-space if and only if for every non-unit $f\in C(X)_\mathcal{P}$, $Z_\mathcal{P}(f)$ is a regular closed set in $X_\mathcal{P}$ (i.e. $cl_{X_\mathcal{P}}(int_{X_\mathcal{P}}(Z_\mathcal{P}(f)))=Z_\mathcal{P}(f)$).
	\begin{lemma} \label{l1}
		Every neighbourhood of a point $x\in X_\mathcal{P}$ contains a zero set neighbourhood  of $x$ in $X_\mathcal{P}$. In other words if $x\in U$, where $U$ is open in $X_\mathcal{P}$, there exists $Z\in Z_\mathcal{P}[X]$ such that $x\in int_{X_\mathcal{P}}Z\subseteq Z\subseteq U$. 
	\end{lemma}
	\begin{proof}
		Let $x\in U$, where $U$ is open in $X_\mathcal{P}$. Then there exists a $g_1\in C(X)_\mathcal{P}$ such that $x\in coz(g_1)\subseteq U$ and $g_1(x)=\{\frac{3}{2} \}$. Define $g=(\boldsymbol{1}-g_1)\vee \{ \boldsymbol{0} \}\in C(X)_\mathcal{P}$. Then $x\in \{y\in X\colon g_1(y)>1 \}\subseteq \{y\in X\colon g_1(y)\geq1 \}=Z_\mathcal{P}(g) $. So $x\in int_{X_\mathcal{P}}Z_\mathcal{P}(g)\subseteq Z_\mathcal{P}(g)\subseteq U$.
	\end{proof}

	\begin{lemma} \label{l0}
		A function $f\in C(X)_\mathcal{P}\setminus \{\boldsymbol{0} \}$ is a regular element (for a commutative ring $R$ with unity, an element $a\in R\setminus \{0\}$ is said to be a regular element if it is not a zero divisor) if and only if $int_{X_\mathcal{P}}(Z_\mathcal{P}(f))= \emptyset$.
	\end{lemma}
	\begin{proof}
		This follows from the fact that $f\in C(X)_\mathcal{P}\setminus \{\boldsymbol{0}\}$ is a zero divisor if and only if $fg=\boldsymbol{0}$ for some $g\in C(X)_\mathcal{P}\setminus \{\boldsymbol{0}\}$ if and only if $\emptyset \neq coz(g)\subseteq Z_\mathcal{P}(f)$.
	\end{proof}

	It follows from the above lemma and the definition of almost $\mathcal{P}P$-spaces that a $\tau_\mathcal{P}$-space $(X,\tau,\mathcal{P})$ is an almost $\mathcal{P}P$-space if and only if each element of $C(X)_\mathcal{P}$ is either a unit or a zero-divisor. Such a ring is said to be a classical ring or an almost regular ring.
	
	Clearly, if $C(X)_\mathcal{P}$ contains all characteristic functions of the form $\chi_{\{{a}\}}$ for all $a\in X$, then it is a classical ring and thus in this case $(X,\tau,\mathcal{P})$ is an almost $\mathcal{P}P$-space. In particular $(X,\tau,\mathcal{P}_f)$ ($C(X)_{\mathcal{P}_f}=C(X)_F$ \cite{GGT2018}), $(X,\tau,\mathcal{P}_{nd})$ ($C(X)_{\mathcal{P}_{nd}}=T'(X)$ \cite{A2010}) and $(X,\tau,\mathcal{K})$ ($C(X)_{\mathcal{K}}=C(X)_K$ \cite{DABM}) are examples of an almost $\mathcal{P}P$-spaces.
	
	We establish another lemma which will help us to study almost $\mathcal{P}P$-spaces.
	\begin{lemma} \label{l2}
		Two sets $A$ and $B$ are $\mathcal{P}$-completely separated in $X$ if and only if they are contained in disjoint zero sets. In fact $\mathcal{P}$-completely separated sets in $X$ are contained in disjoint zero set neighbourhoods in $X_\mathcal{P}$.
	\end{lemma}
	\begin{proof}
		This can be proved by closely following the proof of Theorem 1.15 in \cite{GJ1976}.
		
	\end{proof}
	The following theorem characterizes an almost $\mathcal{P}P$-space algebraically.
	
	\begin{theorem} \label{t1.9}
		The following statements are equivalent for a $\tau \mathcal{P}$-space $(X,\tau,\mathcal{P})$.
		\begin{enumerate}
			\item $X$ is an almost $\mathcal{P}P$-space.
			\item Every $z$-ideal in $C(X)_\mathcal{P}$ is a $z^0$-ideal.
			\item Every maximal ideal in $C(X)_\mathcal{P}$ is a $z^0$-ideal.
			\item Every maximal ideal in $C(X)_\mathcal{P}$ consists entirely of zero divisors.
			\item The sum of any two ideals consisting of zero divisors is either $C(X)_\mathcal{P}$ or consists of zero divisors.
			\item For each non-unit element $f\in C(X)_\mathcal{P}$, there exists a non-zero $g\in C(X)_\mathcal{P}$ such that $P(f)\subseteq Ann(g)$.
		\end{enumerate}
	\end{theorem}
	\begin{proof}
		$(1)\implies(2)$ follows from Theorem \ref{t1.3}, Theorem \ref{t1.6} and the discussion following Definition \ref{d1.8}. $(2)\implies (3)$ is evident from the fact that all maximal ideals are $z$-ideals. It is clear from the definition of $z^0$-ideals that a $z^0$-ideal consists entirely of zero divisors. This proves $(3)\implies (4)$. $(4)\implies (5)$ is obvious. In order to show that $(5)\implies (1)$, let $f\in C(X)_\mathcal{P}$ be a non-unit element. We need to show that $int_{X_\mathcal{P}}(Z_\mathcal{P}(f))\neq \emptyset$. If $Z_\mathcal{P}(f)=X$ or $X\setminus \{x\}$ for some $x\in X$, then we are done. Let $x,y\notin Z_\mathcal{P}(f)$ such that $x\neq y$. Then $x\notin Z_\mathcal{P}(f)\cup \{y\}$. By Lemma \ref{l1} there exists $g\in C(X)_\mathcal{P}$ such that $g\geq \boldsymbol{0}$ and $x\in int_{X_\mathcal{P}}(Z_\mathcal{P}(g))\subseteq Z_\mathcal{P}(g)\subseteq X\setminus (Z_\mathcal{P}(f)\cup \{y\})$. Again $y\notin Z_\mathcal{P}(f)\cup Z_\mathcal{P}(g)$ implies that there exists $h\in C(X)_\mathcal{P}$ such that $h\geq \boldsymbol{0}$ and $x\in int_{X_\mathcal{P}}(Z_\mathcal{P}(h))\subseteq Z_\mathcal{P}(h)\subseteq X\setminus (Z_\mathcal{P}(f)\cup Z_\mathcal{P}(g))$. Thus $fg$ and $fh$ are non-zero elements in $C(X)_\mathcal{P}$ such that $fgh=\boldsymbol{0}$. Therefore the ideals $<fg>$ and $<fh>$ are non-zero ideals consisting of zero divisors. Since $f$ is not a unit, $<fg>+<fh>\neq C(X)_\mathcal{P}$. By $(5)$, $fg+fh$ is a zero divisor. So $int_{X_\mathcal{P}}(Z_\mathcal{P}(f))=int_{X_\mathcal{P}}(Z_\mathcal{P}(fg+fh))\neq \emptyset$. This proves $(1)$. To show that $(1)\implies (6)$, let $f\in C(X)_\mathcal{P}$ be a non-unit element. Then by $(1)$, there exists $g\in C(X)_\mathcal{P}\setminus \{\boldsymbol{0} \}$ such that $fg=\boldsymbol{0}$. Thus $g\in Ann(f)$. It follows from Lemma \ref{l1.4} that $P(f)\subseteq Ann(g)$. Finally we show that $(6)\implies (1)$. Let $f\in C(X)_\mathcal{P}$ be a non-unit element. We need to show that $int_{X_\mathcal{P}}(Z_\mathcal{P}(f))\neq \emptyset$. It is enough to show that $f$ is a zero divisor (by Lemma \ref{l0}). By $(6)$, $P(f)\subseteq Ann(g)$ where $g$ is a non-zero element in $C(X)_\mathcal{P}$. This ensures that $fg=\boldsymbol{0}$.
	\end{proof}
	
	\subsection{Essential ideals of $C(X)_\mathcal{P}$.}
	
	We characterise essential ideals of $C(X)_\mathcal{P}$ by using Lemma \ref{l3.0} and the space $X_\mathcal{P}$. The following theorem works as an extension to Theorem 3.1 of \cite{A1995}.
	
	\begin{theorem} \label{t3}
		For a non-zero ideal $E$ of $C(X)_\mathcal{P}$, the following are equivalent.
		\begin{enumerate}
			\item $E$ intersects every non-zero $z$-ideal of $C(X)_\mathcal{P}$ non-trivially.
			\item $E$ is an essential ideal of $C(X)_\mathcal{P}$.
			\item $int_{X_\mathcal{P}} \bigcap Z_\mathcal{P} [E]=\emptyset $ (that is, $\bigcap Z_\mathcal{P} [E]$ is nowhere dense in $X_\mathcal{P}$ ).
		\end{enumerate}
	\end{theorem}
	\begin{proof}
		(1)$\iff $(2): Let $(1)$ hold and let $I$ be a non-zero ideal of $C(X)_\mathcal{P}$. Then $I_z=Z_\mathcal{P}^{-1}(Z_\mathcal{P}[I])$ is a $z$-ideal of $C(X)_\mathcal{P}$. By (1), there exists $f\in E\cap I_z$. So there exists $g\in I$ such that $Z_\mathcal{P}(f)=Z_\mathcal{P}(g)$. Thus $0\neq fg\in I\cap E$. The converse is obvious
		
		(2)$\iff $(3): Let $ int_{X_\mathcal{P}} \bigcap Z_\mathcal{P} [E]\neq \emptyset $. Then there exists $x_0\in  int_{X_\mathcal{P}} \bigcap Z_\mathcal{P} [E]$, which is open in $X_\mathcal{P}$. So there exists $g\in C(X)_\mathcal{P}$ such that $x_0\in coz(g)\subseteq int_{X_\mathcal{P}} \bigcap Z_\mathcal{P} [E] \subseteq  \bigcap Z_\mathcal{P} [E]$. So for every $f\in E$, $fg=\boldsymbol{0}$ which implies that $\boldsymbol{0}\neq g\in Ann(E)$. By Lemma \ref{l3.0}, $E$ is not an essential ideal of $C(X)_\mathcal{P}$. Conversely, let $ int_{X_\mathcal{P}} \bigcap Z_\mathcal{P} [E]= \emptyset $ and let $I$ be a non-zero ideal of $C(X)_\mathcal{P}$. Then there exists $g\in I\setminus \{\boldsymbol{0} \}$. So $coz(g)\nsubseteq \bigcap Z_\mathcal{P} [E]$ which implies that there exists $f\in E$ such that $coz(g)\nsubseteq Z_\mathcal{P}(f)\implies fg\neq \boldsymbol{0}$. Thus $\boldsymbol{0}\neq fg\in I\cap E$.

	\end{proof}
	
	It is obvious from the above theorem that free ideals of $C(X)_\mathcal{P}$ are essential ideals. Further it can be seen that for $f\in C(X)_\mathcal{P}$, $<f>$ is essential if and only if $f$ is a non-zero, non-unit, regular element. Thus, if $\mathcal{P}$ contains all singleton subsets of $X$, $<f>$ is non-essential for any $f\in C(X)_\mathcal{P}\setminus \{\boldsymbol{0} \}$. In particular, this is true for the rings $C(X)_F$, $T'(X)$ and $C(X)_K$. Further, it follows from the fact that $C(X)_\mathcal{P}$ is a $pm$-ring \cite{DABM} that if $P$ is a prime ideal of $C(X)_\mathcal{P}$, $\bigcap Z_\mathcal{P}[P]$ is either a singleton set or an empty set. (A commutative ring, $R$ with unity is a $pm$-ring if every prime ideal of $R$ is contained in a unique maximal ideal of $R$.) If $P$ is non-essential, then it follows from the above theorem that $\bigcap Z_\mathcal{P}[P]=\{x_0\}$ for some $x_0\in X$ which is an isolated point in $X_\mathcal{P}$. So there exists $g\in C(X)_\mathcal{P}$ such that $\emptyset\neq coz(g)\subseteq \{x_0\}$ which implies that $g=g(x_0)\chi_{\{{x_0}\}}$ and $\chi_{\{{x_0}\}}\in C(X)_\mathcal{P}$. Thus $P\subseteq M_{x_0}=<e>$ where $e=\boldsymbol{1}-\chi_{\{{x_0}\}}$. Since $P$ is a prime ideal, $e(\boldsymbol{1}-e)=\boldsymbol{0}$ and $\boldsymbol{1}-e\notin M_{x_0}$, $e\in P$. Thus $P=M_{x_0}=<e>$, which is also a minimal prime ideal. We summarize these discussions in the following corollary.
	
	\begin{corollary} \label{c3}
		A prime ideal of $C(X)_\mathcal{P}$ is either an essential ideal or a maximal ideal which is also a minimal prime ideal.
	\end{corollary}
	
	Furthermore, if $X$ is a finite set, then $C(X)_\mathcal{P}$ has no free (proper) ideals and for each $x\in X$, $\{x\}$ is open in $X$, and hence open in $X_\mathcal{P}$. So, for any ideal $I$ of $C(X)_\mathcal{P}$, $int_{X_\mathcal{P}}\bigcap Z_\mathcal{P}[I]\neq \emptyset\implies I$ is a non-essential ideal of $C(X)_\mathcal{P}$.
	
	Again, if each ideal of $C(X)_\mathcal{P}$ is non-essential, then $M_x$ (See \cite{DABM}) is non-essential, for each $x\in X$ and so $\{x\}=\bigcap Z_\mathcal{P}[M_x]$ is open in $X_\mathcal{P}$. If $X$ is an infinite set, then $I=\{f\in C(X)_\mathcal{P}\colon coz_\mathcal{P}(f) \text{ is finite.} \}$ is a free (proper) ideal of $C(X)_\mathcal{P}$. Indeed, if not, then there exists $x_0\in \bigcap Z_\mathcal{P}[I]$ and since $\{x_0\}$ is open in $X_\mathcal{P}$, there exists $h\in C(X)_\mathcal{P}$ such that $x_0\in coz_\mathcal{P}(h)\subseteq \{x_0\}$ which ensures that $h\in I$ but $x_0\notin Z_\mathcal{P}(h)$, which is a contradiction. Hence $I$ is an essential ideal of $C(X)_\mathcal{P}$, which contradicts our assumption. This establishes the following corollary.
	
	\begin{corollary} \label{c3.1}
		Each ideal of $C(X)_\mathcal{P}$ is non-essential if and only if $X$ is a finite set.
	\end{corollary}
	
	 We provide a result stronger than Theorem \ref{l1.0} using essential ideals of $C(X)_\mathcal{P}$ which can be proved by following closely the proof of Proposition 3.7 of \cite{A1995}.
	
	\begin{corollary} \label{c3.2}
		A $\tau \mathcal{P}$-space $(X,\tau,\mathcal{P})$ is a $\mathcal{P}P$-space if and only if every essential ideal of $C(X)_\mathcal{P}$ is a $z$-ideal.
	\end{corollary}
	
	The following result, established by Azarpanah in \cite{A1995}, shows how rarely the set of all essential ideals and the set of free ideals of $C(X)$ coincide.
	\begin{theorem} \label{t2} 
		The set of all essential ideals of $C(X)$ and the set of free ideals of $C(X)$ coincide if and only if $X$ is a discrete space.
	\end{theorem}
	\begin{remark} \label{rem}
		 Note that $X_\mathcal{P}$ is the discrete space whenever $C(X)_\mathcal{P}$ contains all characteristic functions of the form $\chi_{\{{x}\}}$ for all $x\in X$. It follows from Theorem \ref{t3} that in this case, the set of all essential ideals always coincides with the set of all free ideals. In particular this property can be seen in rings $C(X)_F$ (\cite{GGT2018}) and $T'(X)$ (\cite{A2010}). 
	\end{remark}

	\section{When is $C(X)_\mathcal{P}$ clean?}
	
	W. K. Nicholson defined a commutative ring $R$ with unity to be clean if every element in $R$ can be expressed as the sum of a unit and an idempotent in $R$ in \cite{N1977} and an element $a\in R$ to be a clean element if it can be expressed as the sum of a unit and an idempotent in $R$ in \cite{N2001}.
	
	It is noted by Anderson and Camillo that for a commutative ring $R$ with finitely many minimal prime ideals, $R$ is a clean ring if and only if it is a $pm$-ring \cite[Theorem 5]{AC2002}. In \cite{DABM} we have shown that $C(X)_\mathcal{P}$ is a $pm$-ring (Corollary 2.16). Further by Theorem 5.8 of \cite{DABM}, if $\mathcal{P}$ contains all singleton subsets of $X$, then the only minimal ideals of $C(X)_\mathcal{P}$ are of the form $<\chi_{\{{a}\}}>$ for $a\in X$, where $\chi_{\{{a}\}}(x)=\begin{cases}
		0 \text{ when }x\neq a \\ 1\text{ when }x=a
	\end{cases}$ for all $x\in X$. But if $|X|\geq 3$ and $\mathcal{P}$ contains all singleton subsets of $X$, then $<\chi_{\{{a}\}}>$ is not a prime ideal for any $a\in X$. Thus, if $|X|\geq 3$ and $\mathcal{P}$ contains all singleton subsets of $X$, then $C(X)_\mathcal{P}$ has no minimal prime ideals. Thus, under these conditions $C(X)_\mathcal{P}$ is a clean ring. This is summarised in the following theorem.

	\begin{theorem} \label{t2.1}
		Let $|X|\geq 3$ and $\mathcal{P}$ contains all singleton subsets of $X$, then $C(X)_\mathcal{P}$ is a clean ring.
	\end{theorem}
	It follows from Theorem \ref{t2.1} that if $|X|\geq 3$, $C(X)_F$ is a clean ring. Also, $T'(X)=C(X)_{\mathcal{P}_{nd}}$ (See \cite{DABM}) contains all characteristic functions of the form $\chi_{\{{a}\}}$ for all $a\in X$ and thus it is a clean ring, even though $\mathcal{P}_{nd}(=$the set of all closed nowhere dense subsets of $X$) may not contain all singleton subsets of $X$.
	
	
	\begin{definitions} \label{d2.2}
		A subset $U\subseteq X$ is said to be $\tau \mathcal{P}$-clopen if $U=Z_\mathcal{P}(f)=X\setminus Z_\mathcal{P}(g)$ for some $f,g\in C(X)_\mathcal{P}$. In other words, by a $\tau \mathcal{P}$-clopen set in $X$ we mean a subset of $X$ which is clopen in $X_\mathcal{P}$.
		
		A $\tau\mathcal{P}$-space $X$ is said to be $\tau\mathcal{P}$-zero-dimensional if for every pair of $\mathcal{P}$-completely separated sets $A,B\subseteq X$, there exists a $\tau\mathcal{P}$-clopen set $U\subseteq X$ such that $A\subseteq U\subseteq X\setminus B$.
	\end{definitions}
	If a subset $U$ of $X$ is such that $U=Z_\mathcal{P}(f)=X\setminus Z_\mathcal{P}(g)$ for some $f,g\in C(X)_\mathcal{P}$, then $U=Z_\mathcal{P}(e)=X\setminus Z_\mathcal{P}(\boldsymbol{1}-e)$, where $e=\frac{|f|}{|f|+|g|}$ and $\boldsymbol{1}-e$ are idempotents in $C(X)_\mathcal{P}$. This gives us the following observation.
	\begin{observation} \label{o2.3}
		A subset $U$ of $X$ is $\tau\mathcal{P}$-clopen if and only if there exists an idempotent $e\in C(X)_\mathcal{P}$ such that $U=Z_\mathcal{P}(e)=X\setminus Z_\mathcal{P}(\boldsymbol{1}-e)$.
	\end{observation}

	We note that $\mathbb{R}$ with the usual topology is not $\tau \mathcal{P}$-zero-dimensional if we take  $\mathcal{P}=\{\emptyset, \{0\} \}$.
	\begin{example} \label{eg1}
		Let us consider $X=\mathbb{R}$ and $\mathcal{P}=\{\emptyset, \{0\} \}$. Then the function $\displaystyle{f(x)=\begin{cases}
				\sin \frac{\pi}{x} &\text{ if }x\neq 0\\
				0 &\text{ otherwise}
			\end{cases}}$ is a member of $C(X)_\mathcal{P}$. However since the only clopen sets in $\mathbb{R}\setminus \{0\}$ are $\mathbb{R}\setminus \{0\}$, $\emptyset$, $(-\infty,0)$ and $(0,\infty)$, there does not exist any idempotent $e\in C(X\setminus \{0\})$ such that $A\setminus \{0\}\subseteq Z(e)\subseteq (X\setminus B)\setminus \{0\}$. Thus there does not exist any idempotent $e'\in C(X)_\mathcal{P}$ such that $A\subseteq Z(e)\subseteq (X\setminus B) $ which ensures that $X$ is not $\tau\mathcal{P}$-zero-dimensional.
		\\ More over, one can also easily see that $\mathbb{R}$ is not $\tau\mathcal{P}$-zero-dimensional if we take  $\mathcal{P}=$the set of all finite subsets of $\mathbb{R}$.

	\end{example}
	
	We look for some examples of $\tau\mathcal{P}$-zero dimensional spaces.
	We know that a Tychonoff space is said to be strongly zero-dimensional \cite{E1977} if two disjoint zero sets can be separated by a clopen set. Clearly, strongly zero-dimensional spaces are $\tau\mathcal{P}$-zero dimensional with $\mathcal{P}=\{\emptyset \}$.
	We use the next observation to achieve a few non-trivial examples (with various choices of $\mathcal{P}$).
	\begin{observation} \label{o2}
		A $C^*$-embedded subspace of a strongly zero-dimensional Tychonoff space is also strongly zero-dimensional.
	\end{observation} 

	It is known that a Tychonoff, extremally disconnected space is strongly zero-dimensional (Theorem 6.2.25 \cite{E1977}). Not only this, but a Tychonoff, extremally disconnected space is $\tau \mathcal{P}$-zero-dimensional for any choice of $\mathcal{P}$ as is seen in the next theorem.
	\begin{theorem} \label{t2.3}
		A Tychonoff, extremally disconnected space  is $\tau \mathcal{P}$-zero-dimensional.
	\end{theorem}
	
	\begin{proof} 
		Let $X$ be a Tychonoff, extremally disconnected space and let $A,B\subseteq X$ be $\mathcal{P}$-completely separated. Then there exists $f\in C^*(X)_\mathcal{P}$ such that $f(A)=\{0\}$ and $f(B)=\{1\}$. It follows that $A\setminus \overline{D_f}$ and $B\setminus \overline{D_f}$ are completely separated in $X\setminus \overline{D_f}$. Since $X\setminus \overline{D_f}$ is an open subspace of $X$ and $X$ is extremally disconnected, by Observation \ref{o2}, there exists an idempotent $e\in C(X\setminus \overline{D_f})$ such that $A\setminus \overline{D_f}\subseteq Z(e)\subseteq (X\setminus \overline{D_f})\setminus (B\setminus \overline{D_f})$. Define $e_0(x)=\begin{cases}
			e(x) \text{ when }x\in X\setminus \overline{D_f}, \\ 0 \text{ when } x\in A\cap \overline{D_f}, \\ 1 \text{ otherwise}
		\end{cases}.$ Since $X\setminus \overline{D_f}$ is open, $D_{e_0}\subseteq \overline{D_f}$. Therefore $e_0\in C^*(X)_\mathcal{P}$ is an idempotent and $A\subseteq Z_\mathcal{P}(e_0)\subseteq X\setminus B$.
	\end{proof}
	
	Another well-known result is that: a Tychonoff space $X$ is extremally disconnected space if and only if $\beta X$ is extremally disconnected. Thus we have the following corollary and example.
	
	\begin{corollary} \label{c2.1}
		If $X$ is a Tychonoff, extremally disconnected space, then both $X$ and $\beta X$ are $\tau \mathcal{P}$-zero-dimensional, for any choice of ideal, $\mathcal{P}$ of closed sets.
	\end{corollary}
	\begin{example} \label{c2.2}
	 $\beta \mathbb{N}$ is $\tau \mathcal{P}$-zero-dimensional, for any choice of ideal, $\mathcal{P}$ of closed sets.
	\end{example}
	
	We now state a characterization of clean elements in $C(X)_\mathcal{P}$.
	\begin{theorem} \label{t2.4}
		An element $f\in C(X)_\mathcal{P}$ is clean if and only if there exists a $\tau\mathcal{P}$-clopen subset $U$ of $X$ such that $f^{-1}(\{1\})\subseteq U\subseteq X \setminus Z_\mathcal{P}(f)$.
	\end{theorem}
	\begin{proof}
		This can be proved by following the proof of Lemma 2.1 \cite{A2002}.
	\end{proof}
	The following corollary is immediate.
	\begin{corollary} \label{c2.5}
		An element $f\in C(X)_\mathcal{P}$ is clean if and only if there exists a $\tau\mathcal{P}$-clopen set $U\subseteq X$ such that either $Z_\mathcal{P}(f)\subseteq U\subseteq X\setminus Z_\mathcal{P}(\boldsymbol{1}-f)$ or $Z_\mathcal{P}(\boldsymbol{1}-f)\subseteq U\subseteq X\setminus Z_\mathcal{P}(f)$.
	\end{corollary}

	Throughout most of this paper, we shall use the above result to recognize clean elements in $C(X)_\mathcal{P}$. It is obvious that every unit in $C(X)_\mathcal{P}$ is clean. Also, for every idempotent $e\in C(X)_\mathcal{P}$, $e=(\boldsymbol{1}-e)+(2e-\boldsymbol{1})$, where $(\boldsymbol{1}-e)$ is an idempotent in $C(X)_\mathcal{P}$ and $(2e-\boldsymbol{1})$ is a unit in $C(X)_\mathcal{P}$. So, every idempotent in $C(X)_\mathcal{P}$ is clean. It is obvious that if for $f\in C(X)_\mathcal{P}$ is such that $f^{-1}(\{1\})=\emptyset$, then $f$ is clean. Thus if $|f|<\boldsymbol{1}$ or $|f|>\boldsymbol{1}$, then $f$ is clean. Moreover, for a unit $f\in C(X)_\mathcal{P}$ with the property that it is not an onto function, $rf$ is a clean element in $C(X)_\mathcal{P}$ for some $r\in \mathbb{R}$.	Not only this, for a given clean element $f\in C(X)_\mathcal{P}$, we have $|f|^r$ is clean for all $r>0$. Also, $\frac{f^2}{f^2+\boldsymbol{1}}$ is clean for any $f\in C(X)_\mathcal{P}$.
	
	It is observed for $\mathcal{P}=\{\emptyset \}$ in Remark 2.3 of \cite{A2002} that the sum or product of two clean elements may not be clean. However we claim that the product of an idempotent and a clean element in $C(X)_\mathcal{P}$ is clean. Indeed, let $e,f\in C(X)_\mathcal{P}$ where $e$ is an idempotent and $f$ is a clean element in $C(X)_\mathcal{P}$. So, by Theorem \ref{t2.4} and Observation \ref{o2.3}, there exists an idempotent $e'\in C(X)_\mathcal{P}$ such that $f^{-1}(\{1\})\subseteq Z_\mathcal{P}(e')\subseteq X\setminus Z_\mathcal{P}(f)$. From this we have $(ef)^{-1}(\{1\})\subseteq Z_\mathcal{P}(e_0)\subseteq X\setminus Z_\mathcal{P}(ef)$, where $e_0=e'(\boldsymbol{1}-e)$ is an idempotent in $C(X)_\mathcal{P}$. It follows from Theorem \ref{t2.4} and Observation \ref{o2.3} that $ef$ is a clean element in $C(X)_\mathcal{P}$. 
	
	As we can see, Theorem \ref{t2.4} is a very useful tool to recognize clean elements in $C(X)_\mathcal{P}$. However, the analog to this result fails in case of $C^*(X)_\mathcal{P}$ if we take $X=(0,1)$ endowed with the subspace topology inherited from $\mathbb{R}$ with usual topology and $f(x)=x$ for all $x\in X$ and $\mathcal{P}=\{\emptyset \}$, then the only idempotents in $C^*(X)_\mathcal{P}$ are $\boldsymbol{1}$ and $\boldsymbol{0}$. So if $f$ is the sum of a unit $u\in C^*(X)_\mathcal{P}$ and an idempotent in $C^*(X)_\mathcal{P}$, then either $u(x)=x-1$ for all $x\in X$ or $u(x)=x$ for all $x\in X$, neither of which are units in $C^*(X)_\mathcal{P}$. Thus $f$ is not a clean element in $C^*(X)_\mathcal{P}$ even though $f^{-1}(\{1\})=\emptyset$ which is a $\tau \mathcal{P}$-clopen subset of $X\setminus Z_\mathcal{P}(f)$. The necessary part of Theorem \ref{t2.4}, however, works even for clean elements in $C^*(X)_\mathcal{P}$. We have the next sufficient condition that ensures that an element $f\in C^*(X)_\mathcal{P}$ is clean.
	
	\begin{theorem} \label{t2.6}
		Let $f\in C^*(X)_\mathcal{P}$ and define for each $a\in \mathbb{R}$, $A_a=\{x\in X\colon f(x)\geq a \}$. If there exist $a,b\in \mathbb{R}$ with the  following properties that \begin{enumerate}
			\item $0<a<b<1$, and
			\item there exists a $\tau \mathcal{P}$-clopen subset $U$ of $X$ such that $A_b\subseteq U\subseteq A_a$,
		\end{enumerate} then $f$ is a clean element of $C^*(X)_\mathcal{P}$.
	\end{theorem}
	\begin{proof}
		By Observation \ref{o2.3}, there exists $e\in C^*(X)_\mathcal{P}$ such that $U=Z_\mathcal{P}(e)$. We define $u(x)=f(x)-e(x)$. We show that $u$ is a unit in $C^*(X)_\mathcal{P}$. Then for $x\in U$, $x\in A_a$ and so $u(x)\geq a$. Again if $x\notin U$, then $x\notin A_b$ and so $u(x)<b-1$. Using the fact that $f$ is bounded on $X$ and hence $u$ is bounded on $X$. Further $|u(x)|\geq \max\{a,1-b\}$ for all $x\in X$. Thus $u$ is a unit in $C^*(X)_\mathcal{P}$ and so $f$ is clean.
	\end{proof}
	
	F. Azarpanah established in \cite{A2002} that $C(X)$ is clean if and only if $X$ is a strongly zero-dimensional space. We attempt to achieve an analogous version of this result in our setting.
	
	\begin{theorem} \label{t2.7}
		The following are equivalent for a $\tau \mathcal{P}$-space $(X,\tau, \mathcal{P})$.
		\begin{enumerate}
			\item \label{2.71} $C(X)_\mathcal{P}$ is a clean ring.
			\item \label{2.72} $C^*(X)_\mathcal{P}$ is a clean ring.
			\item \label{2.73} The set of clean elements in $C(X)_\mathcal{P}$ is a subring of $C(X)_\mathcal{P}$.
			\item \label{2.74} $X$ is a $\tau \mathcal{P}$-zero-dimensional space.
			\item \label{2.75} Every zero divisor of $C(X)_\mathcal{P}$ is clean.
			\item \label{2.76} $C(X)_\mathcal{P}$ has a clean prime ideal.
		\end{enumerate}
	\end{theorem}
	\begin{proof}
		(\ref{2.71})$\iff$ (\ref{2.72}): Let us first assume that $C(X)_\mathcal{P}$ is a clean ring and $f\in C^*(X)_\mathcal{P}$. Define $A=\{x\in X\colon f(x)\geq \frac{2}{3}  \}$ and $B=\{x\in X\colon f(x)\leq \frac{1}{3} \}$. Then $A$ and $B$ are disjoint zero sets in $X$ and so are $\mathcal{P}$-completely separated sets. So there exists $g\in C(X)_\mathcal{P}$ such that $g(A)=\{1\}$ and $g(B)=\{0\}$. By Theorem \ref{t2.4}, there exists a $\tau \mathcal{P}$-clopen subset $U$ of $X$ such that $g^{-1}(\{1\})\subseteq U\subseteq X\setminus Z_\mathcal{P}(g)$. Also, $A_{\frac{2}{3}}=A\subseteq g^{-1}(\{1\})$ and $X\setminus Z_\mathcal{P}(g)\subseteq X\setminus B\subseteq A_{\frac{1}{3}}$. By Theorem \ref{t2.6}, $f$ is clean. Conversely, we assume that $C^*(X)_\mathcal{P}$ is a clean ring. Let $f\in C(X)_\mathcal{P}$ and define $g=(\boldsymbol{-1}\vee f)\wedge \boldsymbol{1}\in C^*(X)_\mathcal{P}$. By (\ref{2.72}) $g$ is clean. Also, by Theorem \ref{t2.4} and using $f^{-1}(\{1\})\subseteq g^{-1}(\{1\})$ and $Z_\mathcal{P}(f)=Z_\mathcal{P}(g)$, it is clear that $f$ is clean. This ensures that the first two assertions are equivalent. 
		
		(\ref{2.71}) $\implies$ (\ref{2.73}) is obvious. Now, we assume (\ref{2.73}) to be true and let $A,B\subseteq X$ be two $\mathcal{P}$-completely separated subsets of $X$. Then there exists $f\in C(X)_\mathcal{P}$ such that $|f|\leq \frac{1}{2}$, $f(A)=\{0\}$ and $f(B)=\{\frac{1}{2}\}$. Since $f^{-1}(\{1\})=\emptyset$, $f$ is a clean element in $C(X)_\mathcal{P}$. 
		
		(\ref{2.71})$\iff$(\ref{2.74}): Let $(1)$ hold. Let $A$ and $B$ be two $\mathcal{P}$-completely separated subsets of $X$. Then there exists $g\in C(X)_\mathcal{P}$ such that $g(A)=\{0\}$ and $g(B)=\{1\}$. By (\ref{2.71}), $g$ is a clean element. It follows from Theorem \ref{t2.4} that there exists a $\tau \mathcal{P}$-clopen set $U$ such that $g^{-1}(\{1\})\subseteq U\subseteq X\setminus Z_\mathcal{P}(g)$. But $A\subseteq Z_\mathcal{P}(g)$ and $B\subseteq g^{-1}(\{1\})$. This ensures that $X$ is a $\tau \mathcal{P}$-zero-dimensional space. Conversely, let $X$ be a $\tau \mathcal{P}$-zero-dimensional space and $f\in C(X)_\mathcal{P}$. Then we need to show that there exists a $\tau \mathcal{P}$-clopen set $U$ such that $f^{-1}(\{1\})\subseteq U\subseteq X\setminus Z_\mathcal{P}$. We define $A=\{x\in X\colon f(x)\geq \frac{1}{2} \}$ and $B=Z_\mathcal{P}(f)$. Then $A$ and $B$ are disjoint zero sets and are thus $\mathcal{P}$-completely separated sets. The rest follows directly since $f^{-1}(\{1\})\subseteq A$. This proves (\ref{2.74})$\implies $(\ref{2.71}).
		
		Again (\ref{2.71})$\implies$(\ref{2.75}) is obvious. 
		
		We next show that (\ref{2.75})$\implies$(\ref{2.76}). Let every zero divisor of $C(X)_\mathcal{P}$ be clean. If $P$ is a minimal prime ideal of $C(X)_\mathcal{P}$. Then it follows from Lemma \ref{l1.4} that each $f\in P$ is a zero divisor, as $P(f)\subseteq P$. Thus $P$ is a clean prime ideal. If $C(X)_\mathcal{P}$ does not contain any minimal prime ideals, then it follows from Theorem 5 in \cite{AC2002} that $C(X)_\mathcal{P}$ is a clean ring. 
		
		Finally, we need to show that (\ref{2.76})$\implies$(\ref{2.71}). By (\ref{2.76}), there exists a clean prime ideal $P$ of $C(X)_\mathcal{P}$. Let $f\in C(X)_\mathcal{P}$. Define $A=Z_\mathcal{P}(f)$ and $B=f^{-1}(\{1\})$. It follows from the discussions following Corollary \ref{c2.5} that if either $A=\emptyset$ or $B=\emptyset$, then we are done. So we assume that $A$ and $B$ are both non-empty sets. Set two zero sets $C=\{x\in X\colon f(x)\leq \frac{1}{3} \}$ and $D=\{x\in X\colon f(x)\geq \frac{1}{3} \}$. Since $A$ and $D$ are disjoint zero sets, they are $\mathcal{P}$-completely separated. Analogously, $B$ and $C$ are pairs of $\mathcal{P}$-completely separated sets. Therefore there exists $g,h\in C(X)_\mathcal{P}$ such that $g(A)=\{1\}$, $g(D)=\{0\}$, $h(B)=\{1\}$ and $h(C)=\{0\}$. Clearly $gh=\boldsymbol{0}\in P$. So either $g\in P$ or $h\in P$ which implies that either $g$ or $h$ is clean. If $g$ is clean, then there exists a $\tau \mathcal{P}$-clopen set $U$ such that $g^{-1}(\{1\})\subseteq U\subseteq X\setminus Z_\mathcal{P}(g)$. But $A\subseteq g^{-1}(\{1\})$ and $B\subseteq D\subseteq Z_\mathcal{P}(g)$. Again if $h$ is clean, then there exists a $\tau \mathcal{P}$-clopen set $V$ such that $h^{-1}(\{1\})\subseteq V\subseteq X\setminus Z_\mathcal{P}(h)$. But $B\subseteq g^{-1}(\{1\})$ and $A\subseteq C\subseteq Z_\mathcal{P}(h)$. It follows from Corollary \ref{c2.5} that $f$ is clean.
	\end{proof} 

	The following example follows directly from the above theorem and Example \ref{c2.2}.
	\begin{example} \label{c2.3}
		$C(\beta \mathbb{N})_\mathcal{P}$ is a clean ring for any choice of ideal $\mathcal{P}$ of closed sets.
	\end{example}
	
	A ring $R$ is said to be an exchange ring if and only if for every $a\in R$ there exist $b,c\in R$ such that $c(1-a)(1-ba)=1-ba$ and $bab=b$. \cite{M1972}.
	It is stated in \cite{A2002} that $C(X)$ is a clean ring if and only if it is an exchange ring. It is natural to ask whether the analog to the above statement holds for $C(X)_\mathcal{P}$. We answer this question in the affirmative.
	
	\begin{proposition} \label{p2.8}
		$C(X)_\mathcal{P}$ is a clean ring if and only if it is an exchange ring.
	\end{proposition}
	\begin{proof}
		Let us assume that $C(X)_\mathcal{P}$ is a clean ring and let $f\in C(X)_\mathcal{P}$. Then there exist $e,u\in C(X)_\mathcal{P}$ such that $e$ is an idempotent and $u$ is a unit. Define $g(x)=\begin{cases}
			0, \text{ when }x\notin Z_\mathcal{P}(e) \\ \frac{1}{f(x)}, \text{ when }x\in Z_\mathcal{P}(e)
		\end{cases}$ and $h(x)=\begin{cases}
			\frac{1}{1-f(x)} \text{ when }x\notin Z_\mathcal{P}(e) \\ 0 \text{ when }x\in Z_\mathcal{P}(e)
		\end{cases}$. Then $g,h\in C(X)_\mathcal{P}$ as $D_g\subseteq D_f$ and $D_h\subseteq D_f$. It is clear that $gfg=g$ and $h(\boldsymbol{1}-f)(\boldsymbol{1}-fg)=\boldsymbol{1}-fg$. This shows that $C(X)_\mathcal{P}$ is an exchange ring.
		
		Conversely, let $C(X)_\mathcal{P}$ be an exchange ring. Then there exists $gh\in C(X)_\mathcal{P}$ such that $gfg=g$ and $h(\boldsymbol{1}-f)(\boldsymbol{1}-fg)=\boldsymbol{1}-fg$. It follows that $Z_\mathcal{P}(\boldsymbol{1}-fg)$ is a $\tau \mathcal{P}$-clopen set in $X$ and $f^{-1}({1})\subseteq U\subseteq X\setminus Z_\mathcal{P}(f)$.
	\end{proof}
	
	The following corollary follows from Theorem \ref{t2.7} and Proposition \ref{p2.8}.
	
	\begin{corollary} \label{c2.9}
		For a $\tau \mathcal{P}$-space $X$, $C(X)_\mathcal{P}$ is an exchange ring if and only if $X$ is a $\tau \mathcal{P}$-zero-dimensional space.
	\end{corollary}
	
	Ye \cite{Ye} defined an element $r$ of ring $R$ to be semiclean if $r=a+u$, where $a$ is a periodic element, that is $a^k=a^l$ for some positive integers $k$ and $l$ such that $k\neq l$ and $u$	is a unit in R. A ring $R$ is called a semiclean ring if every element of $R$ is	semiclean. Further, an element $a\in R$ is said to be weakly clean if and only if $a$ can be expressed as $a=u+e$ or $a=u-e$ where $u$ is a unit and $e$  is an idempotent in $R$. A ring $R$ is said to be a weakly clean ring if every element of $R$ is weakly clean \cite{AA2006}. Clearly a ring $R$ is semiclean if it is weakly clean and $R$ is weakly clean if it is a clean ring. N. Arora and S. Kundu showed in \cite{AK2014} that in $C(X)$, the concept of clean, semiclean and weakly clean are equivalent. We establish a similar equivalence using Observation \ref{o2.3}.
	
	\begin{theorem} \label{t2.10}
		For a $\tau\mathcal{P}$-space $(X,\tau, \mathcal{P})$, the following statements are equivalent.
		\begin{enumerate}
			\item $C(X)_\mathcal{P}$ is clean.
			\item $C(X)_\mathcal{P}$ is weakly clean.	
			\item $C(X)_\mathcal{P}$ is semiclean.		
		\end{enumerate}
	\end{theorem}

	\begin{proof}
		$(1)\implies (2)\implies (3)$ is obvious. We just need to show that $C(X)_\mathcal{P}$ is clean if it is semiclean. Let us assume that $C(X)_\mathcal{P}$ is semiclean. Also let $f\in C(X)_\mathcal{P}$ and $g=2f-\boldsymbol{1}$. Then there exists a unit $u\in C(X)_\mathcal{P}$ and a periodic element $p\in C(X)_\mathcal{P}$ such that $g=u+p$. Define $K=\{ x\in X\colon u(x)<0 \}=\{ x\in X\colon u(x)\leq 0 \}$. Then $K$ is a $\tau \mathcal{P}$-clopen subset of $X$. By Observation \ref{o2.3}, there exists an idempotent element $e\in C(X)_\mathcal{P}$ such that $K=coz(e)=Z_\mathcal{P}(\boldsymbol{1}-e)$. Since $p$ is a periodic elememt in $C(X)_\mathcal{P}$, $p(X)\subseteq \{ 0,1,-1 \}$. This ensures that $u(x)\in \{1,2 \}$ on $Z_\mathcal{P}(\boldsymbol{1}-f)$ and $u(x)\in \{-1,-2 \}$ on $Z_\mathcal{P}(f)$. Thus we have $f-e$ is a unit in $C(X)_\mathcal{P}$. This proves $(1)$.
	\end{proof}

	McGovern introduced the notion of an almost clean element of a ring $R$ in \cite{M2003}. An element in a ring $R$ is said to be an almost clean element
	if it can be expressed as the sum of a regular element (an element which is not a zero divisor) and an idempotent. If every element of $R$ is almost clean, then $R$ is said to be an almost clean ring.  He established that for a Tychonoff space $X$, $C(X)$ is an almost clean ring if and only if it is clean. R. Mohamadian later characterized almost clean elements in $C(X)$ (See \cite{M2022}) and established the same result ($C(X)$ is an almost clean ring if and only if it is clean) using the characterization of almost clean elements. It is clear that if $\mathcal{P}$ contains all singleton subsets of $X$, then each element in $C(X)_\mathcal{P}$ is either a unit or a zero divisor. Hence in this case, $C(X)_\mathcal{P}$ is almost clean if and only if it is clean.

	 \begin{theorem} \label{t2.11}
	 	An element $f\in C(X)_\mathcal{P}$ is an almost clean element if and only if there exists a $\tau \mathcal{P}$-clopen subset $U$ of $X$ such that \[int_{X_\mathcal{P}}(f^{-1}(\{1\})) \subseteq U\subseteq cl_{X_\mathcal{P}}(coz(f)). \]
	 \end{theorem}
	\begin{proof}
		Let $f\in C(X)_\mathcal{P}$ be an almost clean element. So there exist a regular element $r\in C(X)_\mathcal{P}$ and an idempotent $e\in C(X)_\mathcal{P}$ such that $f=r+e$. Thus $Z_\mathcal{P}(f)\cap Z_\mathcal{P}(e)\subseteq Z_\mathcal{P}(r)$, where $int_{X_\mathcal{P}}(Z_\mathcal{P}(r))=\emptyset$ (by Lemma \ref{l0}). This shows that $int_{X_\mathcal{P}}(Z_\mathcal{P}(f))\cap Z_\mathcal{P}(e)=\emptyset$. Using this we get $Z_\mathcal{P}(e)\subseteq cl_{X_\mathcal{P}}(coz(f))$. Again since $f=r+e$ and $int_{X_\mathcal{P}}(Z_\mathcal{P}(r))=\emptyset$ (by Lemma \ref{l0}), we have $int_{X_\mathcal{P}}(f^{-1}(\{1\}))\cap coz(e)=\emptyset$. This shows that $int_{X_\mathcal{P}}(f^{-1}(\{1\})) \subseteq Z_\mathcal{P}(e)$.
		
		Conversely, let there exists a $\tau \mathcal{P}$-clopen subset $U$ of $X$ such that \\ $int_{X_\mathcal{P}}(f^{-1}(\{1\})) \subseteq U\subseteq cl_{X_\mathcal{P}}(coz(f))$. By Observation \ref{o2.3}, there exists an idempotent $e\in C(X)_\mathcal{P}$ such that $U=Z_\mathcal{P}(e)=coz(\boldsymbol{1}-e)$. Let us suppose $A=Z_\mathcal{P}(e)\cap Z_\mathcal{P}(f)$ and $B=f^{-1}(\{1\})\cap coz(e)$. Then $A\cup B=Z_\mathcal{P}(h)$, where $h=(e+|f|)(|\boldsymbol{1}-f|+(1-e))\in C(X)_\mathcal{P}$. Define $r\colon X\longrightarrow \mathbb{R}$ by \[r(x)=\begin{cases}
			0 &\text{, if }x\in Z(h|_{X\setminus \overline{D_h} }) \\ 
			f(x)-e(x) &\text{, if }x\in cl_{ X\setminus \overline{D_h}}(coz(h)) \\
			0 &\text{, if }x\in Z_\mathcal{P}(h|_{\overline{D_h} }) \\ 
			f(x)-e(x) &\text{, otherwise }
		\end{cases}. \] Then $r|_{X\setminus \overline{D_h}}\in C(X\setminus \overline{D_h})$. So $D_r\subseteq \overline{D_h}$ and thus $r\in C(X)_\mathcal{P}$. It is to note that $A$ and $B$ are disjoint closed sets in $X_\mathcal{P}$ such that $int_{X_\mathcal{P}}A=\emptyset$ and $int_{X_\mathcal{P}}B=\emptyset$. Thus we have $int_{X_\mathcal{P}}(A\cup B)=\emptyset$. But $A\cup B=Z_\mathcal{P}(r)$. It follows from Lemma \ref{l0} that $r$ is a regular element and $f=r+e$ is evident.
	\end{proof}
	We have now successfully characterized almost clean elements of the ring $C(X)_\mathcal{P}$. Using this we achieve the following result.
	
	\begin{theorem} \label{t2.12}
		$C(X)_\mathcal{P}$ is clean if and only if it is almost clean.
	\end{theorem}
	\begin{proof}
		Let us assume that $C(X)_\mathcal{P}$ is an almost clean ring and $f\in C(X)_\mathcal{P}$. Let $A=\{x\in X\colon f(x)\geq \frac{2}{3} \}$ and $B=\{x\in X\colon f(x)\leq \frac{1}{3} \}$. Then there exists $g\in C(X)_\mathcal{P}$ such that $g(A)=\{0\}$ and $g(B)=\{1 \}$. By Theorem \ref{t2.11} and Observation \ref{o2.3}, there exists $e\in C(X)_\mathcal{P}$ such that $int_{X_\mathcal{P}}(g^{-1}(\{1\})) \subseteq Z_\mathcal{P}(e)\subseteq cl_{X_\mathcal{P}}(coz(g))$. This ensures that $int_{X_\mathcal{P}}(f^{-1}(\{1\})) \subseteq f^{-1}(\{1\})\subseteq f^{-1}(\frac{2}{3},\infty) \subseteq A\subseteq Z_\mathcal{P}(g)$. Thus we have $int_{X_\mathcal{P}}(f^{-1}(\{1\}))\cap Z_\mathcal{P}(e)=\emptyset$ and $Z_\mathcal{P}(\boldsymbol{1}-e)\subseteq coz(f)$, that is, $int_{X_\mathcal{P}}(f^{-1}(\{1\}))\subseteq Z_\mathcal{P}(\boldsymbol{1}-e)\subseteq coz(f)$, where $\boldsymbol{1}-e$ is an idempotent. Thus $f$ is clean. 
		
		The rest is obvious.
	\end{proof}

	Thus in the ring $C(X)_\mathcal{P}$, the notions of a clean ring, a semiclean ring, a weakly clean ring, an almost clean ring and an exchange ring are  all coincide. Furthermore, $C(X)_\mathcal{P}$ is any of the aforementioned rings if and only if $(X,\tau, \mathcal{P})$ is $\tau \mathcal{P}$-zero-dimensional. 
	
	In particular, $\boldsymbol{C(\beta \mathbb{N})_\mathcal{P}}$ is a clean ring, a semiclean ring, a weakly clean ring, an almost clean ring and an exchange ring for any choice of ideal $\mathcal{P}$ of closed sets (See Example \ref{c2.2}).

	\section{Depth of ideals of $C(X)_\mathcal{P}$}
	
	Let $R$ be a commutative ring with unity and $M$ be an $R$-module. By a $M$-regular element $a$, we mean an $a\in R$ such that $am\neq 0$ for all $m\in M\setminus \{0\}$. A sequence $a_1,a_2,\cdots, a_n\in R $ is called a $M$-regular sequence if $a_1$ is $M$-regular, $a_2$ is $\dfrac{M}{a_1M}$-regular, $a_3$ is $\dfrac{M}{a_1M+a_2M}$-regular and so on; and $\displaystyle{\sum_{i=1}^{n}a_iM\neq M  }$. The length of a maximal $M$-regular sequence is called the depth of $M$. \cite{Eisenbud} 
	
	Thus, for an ideal $I$ of $C(X)_\mathcal{P}$, an element $f\in C(X)_\mathcal{P}$ is $I$-regular if and only if $Ann(f)\cap I=\emptyset$. Clearly each unit in $C(X)_\mathcal{P}$ is $I$-regular for any ideal $I$ of $C(X)_\mathcal{P}$. Also an element $f\in C(X)_\mathcal{P}$ is $\dfrac{I}{gI}$-regular if and only if $fh\notin gI$ for every $h\notin I$.
	
	In order to compute the depths of ideals of $C(X)_\mathcal{P}$, we first need the following lemma.
	
	\begin{lemma} \cite{AES2019} \label{l5.1}
		The following hold for a reduced ring $R$.
		\begin{enumerate}
			\item If $a,b\in R$, then $a$ is $M(b)$-regular if and only if $b\in P(a)$.
			\item If $I$ is a $z$-ideal of $R$, then $a\in R$ is $I$-regular if and only if $I\subseteq P(a)$.
		\end{enumerate}
	\end{lemma}
	The next results follow directly from Lemmas \ref{l1.4}, \ref{l1.5} and \ref{l5.1}.
	\begin{theorem} \label{t5}
		Let $I$ be an ideal of $C(X)_\mathcal{P}$ and $f\in C(X)_\mathcal{P}$. Then the following hold. 
		\begin{enumerate}
			\item For $g\in C(X)_\mathcal{P}$, $f$ is $M(g)$-regular if and only if $int_{X_\mathcal{P}}Z_\mathcal{P}(f)\subseteq int_{X_\mathcal{P}}Z_\mathcal{P}$ if and only if $g\in P(f)$.
			\item $f$ is $I$-regular if and only if it is $I_z$-regular, where $I_z=Z_\mathcal{P}^{-1}(Z_\mathcal{P}[I])$ is a $z$-ideal of $C(X)_\mathcal{P}$.
			\item $f$ is $I$-regular if and only if $int_{X_\mathcal{P}}Z_\mathcal{P}(f)\subseteq \bigcap Z_\mathcal{P}[I]$ if and only if $I\subseteq  P(f)$.
		\end{enumerate}
	\end{theorem}
	It is clear from above that for each $f\in C(X)_\mathcal{P}$, $f$ is $<f>$-regular, $M(f)$-regular and $P(f)$-regular. Also, for $f\in C(X)_\mathcal{P}$,  the collection $\mathcal{R}_f=\{I\colon I\text{ is an ideal of }C(X)_\mathcal{P} \text{ such that }f\text{ is }I\text{-regular} \}$ has a largest element, viz, $P(f)$. Further, we have the following corollary using Theorem \ref{t5} and Lemma \ref{l1.4}.
	\begin{corollary}
		For any $g\in C(X)_\mathcal{P}$, the collection of $M(g)$-regular functions coincide with that of $P(g)$-regular functions. That is, $f\in C(X)_\mathcal{P}$ is $M(g)$-regular if and only  if it is $P(g)$-regular.
	\end{corollary}

	A characterization of an essential ideal $I$ of $C(X)_\mathcal{P}$ can be given using regular and $I$-regular elements of the ring as follows.
	
	\begin{theorem} \label{t5.1.0}
		An ideal $I$ of $C(X)_\mathcal{P}$ is essential if and only if $r_I(X)=r(X)$, where $r_I(X)$ (resp. $r(X)$) is the set of $I$-regular (resp. regular) elements in $C(X)_\mathcal{P}$.
	\end{theorem} 
	\begin{proof}
	 	It is clear that $f\in r_I(X)$ if and only if $I\cap Ann(f)=\{\boldsymbol{0} \}$. So $I$ is an essential ideal if $Ann(f)$ is a non-trivial ideal. Conversely, let $I$ be a non-essential ideal of $C(X)_\mathcal{P}$. Then by Theorem \ref{t3}, there exists $x_0\in int_{X_\mathcal{P}} \bigcap Z_\mathcal{P}[E]$. Following the steps in the proof of Lemma \ref{l1}, there exists $g\in C(X)_\mathcal{P}$ such that $x_0\in int_{X_\mathcal{P}}Z_\mathcal{P}(g)\subseteq Z_\mathcal{P}(g)\subseteq int_{X_\mathcal{P}} \bigcap Z_\mathcal{P}[E]$ and $g(X\setminus int_{X_\mathcal{P}} \bigcap Z_\mathcal{P}[E])=\{\boldsymbol{1} \}$. Thus it follows from Theorems \ref{t5}(3) and Lemma \ref{l0} that $g\in r_I(X)\setminus r(X)$.	
	\end{proof}
	The following theorem gives the depth of certain ideals of $C(X)_\mathcal{P}$.
	\begin{theorem} \label{t5.1}
		The following statements hold for a $\tau \mathcal{P}$-space $(X,\tau , \mathcal{P})$.
		\begin{enumerate}
			\item The depth of each ideal of $C(X)_\mathcal{P}$ is zero if and only if $X$ is a $\mathcal{P}P$-space.
			\item The depth of each essential ideal  of $C(X)_\mathcal{P}$ is zero if and only if $X$ is an almost $\mathcal{P}P$-space.
		\end{enumerate}
	\end{theorem}
	\begin{proof}
		\begin{enumerate}
			\item Let us assume first that $X$ is a $\mathcal{P}P$-space and let $I$ be an ideal of $C(X)_\mathcal{P}$. Then $I$ is a $z$-ideal (by Theorem \ref{l1.0}). Also let $f\in C(X)_\mathcal{P}$ be a non-unit element in $r_I(X)$. Then by Lemma \ref{l01}, $Z_\mathcal{P}(f)=coz_\mathcal{P}(g)$ for some $g\in C(X)_\mathcal{P}$. Define for each $i\in I$, $h_i(x)=\begin{cases}
				\dfrac{i(x)}{f(x)} &\text{, if }x\in Z_\mathcal{P}(g)\\ 0 &\text{, if }x\in Z_\mathcal{P}(f)
			\end{cases}.$ Then $h_i|_{Y}\in C(Y)$, where $Y=X\setminus (\overline{D_f}\cup \overline{D_g})$ is an open subset of $X$ and $\overline{D_f}\cup \overline{D_g}\in \mathcal{P}$. So $h_i\in C(X)_\mathcal{P}$ and $i=h_if$. Thus $I=fI$. Therefore $depth(I)=0$. Conversely let the depth of every ideal of $C(X)_\mathcal{P}$ be zero and let $f\in C(X)_\mathcal{P}$ be a non-unit. Clearly $f$ is $M(f)$-regular and $depth(M(f))=0$. Thus $fM(f)=M(f)$. So there exists $g\in M(f)$ ($\implies$ $Z_\mathcal{P}(f)\subseteq Z_\mathcal{P}(g)$) such that $fg=f$ which implies that $Z_\mathcal{P}(f)=coz(\boldsymbol{1}-g)$ and so by Lemma \ref{l01}, $f$ is Von-Neumann regular. Thus $X$ is a $\mathcal{P}P$-space. 
		\item Let $I$ be an essential ideal of $C(X)_\mathcal{P}$ for an almost $\mathcal{P}P$-space $X$. Then $r_I(X)=r(X)=U(X)$, where $U(X)$ is the collection of all units in $C(X)_\mathcal{P}$. Thus there is no non-unit $I$-regular element in $C(X)_\mathcal{P}$. Thus $depth(I)=0$. Conversely, let us assume that $X$ is not an almost $\mathcal{P}P$-space. Then there exists a non-unit zero divisor $f\in C(X)_\mathcal{P}$. Since $Ann(f)=\{\boldsymbol{0} \}$, $I=<f>$ is essential and so $r_I(X)=r(X)$. Thus $f$ is $I$-regular. If $depth(I)=0$, $fI=I$ which implies that there exists $g\in I$ such that $f=f^2g$ and hence $f(\boldsymbol{1}-fg)=\boldsymbol{0}$ which contradicts that $f$ is a non-unit zero divisor. Thus $depth(I)\neq 0$.
		\end{enumerate}
	\end{proof}

	Note that since $(X,\tau,\mathcal{P}_f)$ and $(X,\tau,\mathcal{P}_{nd})$ are almost $\mathcal{P}P$-spaces, depth of each essential ideal of $C(X)_F$ (or $T'(X)$) is zero (See Remark \ref{rem}). Furthermore, free ideals are always essential and so depth of each free ideal of $C(X)_F$ (or $T'(X)$) is zero.  In fact, if $\mathcal{P}$ is such that $C(X)_\mathcal{P}$ contains the characteristic functions $\chi_{\{{x}\}}$ for all $x\in X$, then the depth of each free ideal (and hence essential ideal) of $C(X)_\mathcal{P}$ is zero. Moreover, in this case, the depth of each maximal ideal of $C(X)_\mathcal{P}$ is also zero. In fact, we get something stronger, that is, depth of each prime ideal of $C(X)_\mathcal{P}$ is zero. These results are summarized as follows.

\begin{theorem} \label{o3}
	If $\mathcal{P}$ is such that $C(X)_\mathcal{P}$ contains the characteristic functions $\chi_{\{{x}\}}$ for all $x\in X$, then
	\begin{enumerate}
		\item the depth of each free ideal (and hence each essential ideal) of $C(X)_\mathcal{P}$ is zero. 
		\item the depth of each maximal ideal of $C(X)_\mathcal{P}$ is zero.
		\item the depth of each prime ideal of $C(X)_\mathcal{P}$ is zero.
	\end{enumerate}
\end{theorem}
\begin{proof}
	\begin{enumerate}
		\item Follows from the above discussions.
		\item Let $M$ be a maximal ideal of $C(X)_\mathcal{P}$. If $M$ is a free ideal, $depth(M)=0$. For a fixed maximal ideal $M$, there exists $x_0\in X$ such that $M=M_{x_0}$ (see \cite{DABM}). Let $f\in C(X)_\mathcal{P}$ be $M$-regular. Then $fg\neq \boldsymbol{0}$ for any $g\in M\setminus \{\boldsymbol{0} \}$. Thus for each $x\in X\setminus \{x_0\}$, $\chi_{\{{x}\}}f\neq \boldsymbol{0}$ and thus $f(x)\neq 0$. Therefore $X\setminus \{x_0 \}\subseteq coz(f)$.
		\begin{enumerate}
			\item[Case 1.] If $f(x_0)=0$, then $f\in M$ and hence $fM=M$.
			\item[Case 2.] If $f(x_0)\neq0$, then $f$ is a unit and so $ fM=M$.
		\end{enumerate} Thus there does not exist any $M$-regular sequence and hence $depth(M)=0$.
		\item Follows from Observations \ref{o3} and Corollary \ref{c3}.
	\end{enumerate}
\end{proof}
	
	In particular, depth of each prime ideal of $C(X)_F$ (and $T'(X)$) is zero.

	\begin{remark}
		\begin{enumerate}
			\item It is important to note that the above observation may fail if $C(X)_\mathcal{P}$ fails to contain the characteristic functions $\chi_{\{{x}\}}$ for all $x\in X$. To see this let $X=[0,1]\cup \{2\}$ endowed with the subspace topology induced from the usual topology of $\mathbb{R}$ and $\mathcal{P}=\{\emptyset, \{2\}\}$. Then $X$ contains non-almost $P$-points (In a Tychonoff space $X$, a point $x\in X$ is said to be an almost $P$-point if every zero set $Z\in Z[X]$ containing $x$ has non-empty interior, i.e. $int_XZ\neq \emptyset$.) and $C(X)_\mathcal{P}=C(X)$. It follows from Corollary 3.9 \cite{AES2019} that the depth of each maximal ideal of $C(X)=C(X)_\mathcal{P}$ is $1$. Further, $C(X)_\mathcal{P}=C(X)$ contains free maximal ideals and so there exists a free ideal (which is also essential) of $C(X)_\mathcal{P}=C(X)$ having depth $1$. Indeed, there also exists a prime ideal of $C(X)_\mathcal{P}=C(X)$ having depth $1$.
			\item In light of Theorem \ref{o3}, we note that the depth of all ideals of $C(X)_\mathcal{P}$ may not be zero even if $C(X)_\mathcal{P}\supseteq \{\chi_{\{x \}}\colon x\in X \}$. For example, take  $X=\mathbb{R}$, the real line. Then  $C(X)_F$ is not a Von-Neumann Regular ring. Indeed, if $f(x)=sin$ $x$ for all $x\in X$, then there does not exist a $g\in C(X)_F$ such that $f=f^2g$. Thus by Theorem \ref{t5.1}, it follows that there exists an ideal of $C(X)_F$ having non-zero depth.
		\end{enumerate}
		 
	\end{remark}
	
	\bibliographystyle{plain}

\end{document}